\documentclass[a4paper,12pt]{extarticle}
\usepackage{amsmath}
\usepackage{amssymb}
\usepackage{amsthm}

\usepackage{hyperref}
\hypersetup{
    colorlinks,
    citecolor=blue,
    filecolor=blue,
    linkcolor=blue,
    urlcolor=blue
}
\allowdisplaybreaks

\begin{document}
\newtheorem{theor}{Theorem}[section] 
\newtheorem{prop}[theor]{Proposition} 
\newtheorem{cor}[theor]{Corollary}
\newtheorem{lemma}[theor]{Lemma}
\newtheorem{sublem}[theor]{Sublemma}
\newtheorem{defin}[theor]{Definition}
\newtheorem{conj}[theor]{Conjecture}
\theoremstyle{definition}
\newtheorem{Bem}[theor]{Remark}
\newtheorem{ex}[theor]{Example}
\hfuzz2cm

\gdef\mtr#1{#1}
\gdef\ar#1{\widehat{#1}}
\def\covol{{\rm covol}}
\def\odd{{\rm odd}}
\gdef\a{\alpha}
\gdef\mn{{\mu_{m}}}
\gdef\Hom{{\rm Hom}}
\def\Trs{{\rm Tr_s\,}}
\def\Tr{{\rm Tr\,}}
\def\End{{\rm End}}
\def\eq{equivariant }
\def\Td{{\rm Td}}
\def\ch{{\rm ch}}
\def\torus{{\cal T}}
\def\Proj{{\rm Proj}}
\def\uexp#1{{{\rm e}^{#1}}}
\def\Spec{{\rm Spec}\,}
\def\mod{{\rm mod}}
\def\rk{{\rm rk}\,}
\def\phi{\varphi}
\def\deg{{{\rm deg}\,}}
\def\Td{{\rm Td}}
\def\ch{{\rm ch}}
\def\Hom{{\rm Hom}}
\def\Bil{{\rm Bil}}
\def\End{{\rm End}}
\def\Sym{\Sym}
\gdef\endProof{\hfill$\square$\par\bigskip\noindent}
\def\Im{{\rm Im}\,}
\def\Re{{\rm Re}\,}
\def\im{{\rm im}\,}
\def\ker{{\rm ker}\,}
\def\Tr{{\rm Tr}\,}
\def\Eig{{\rm Eig}\,}
\def \x{\times}
\def\Spec{{\rm Spec}\,}
\def \BN{{\mathbb N}}
\def \BZ{{\mathbb Z}}
\def \BQ{{\mathbb Q}}
\def \BR{{\mathbb R}}
\def \BC{{\mathbb C}}
\def \BH{{\mathbb H}}
\def\BP{{\mathbb P}}
\def\Sp{{\bf Sp}}
\def\GL{\mbox{\bf GL}}
\def\SL{\mbox{\bf SL}}
\def\SO{\mbox{\bf SO}}
\def\O{\mbox{\bf O}}
\def\SU{\mbox{\bf SU}}
\def\so{{\mathfrak {so}}}
\def\Aut{\mbox{\rm Aut}}
\def\sign{\mbox{\rm sign\,}}
\def\ord{{\rm ord}}
\def \<{\langle}
\def \>{\rangle}
\def \Ad {{\rm Ad}}
\def \ad {{\rm ad}}
\def\vp{\varphi}
\def\vep{\varepsilon}
\def\theta{\vartheta}
\def\Ric{{\rm Ric}}
\def\vol{{\rm vol}}
\def\diam{{\rm diam}\,}
\def\Pf{{\rm Pf}}
\def\grad{{\rm grad}\,}
\def\eps{\varepsilon}
\def\div{{\rm div\,}}
\def\Sym{{\rm Sym}}
\def\Cas{{\rm Cas}}
\def\nm{{n_{\max}}}

\def\bz{\mbox{\boldmath$\zeta$\unboldmath}}
\def\bzs{\mbox{\boldmath$\zeta'$\unboldmath}}
\def\odd{{\rm odd}}
\author{
Kai K\"ohler\footnote{Mathematisches Institut/
Universit\"atsstr. 1/
Geb\"aude 25.22/
D-40225 D\"usseldorf/
koehler@math.uni-duesseldorf.de}} 
\title{Asymptotics of holomorphic torsion on orbifolds and an equivariant arithmetic Hilbert-Samuel theorem}
\maketitle
\begin{abstract}
We compute the $\log p$-part of the asymptotic expansion of the holomorphic torsion of the $p$-th power of a positive line bundle in terms of characteristic classes. This is done for orbifold and equivariant torsions. In the equivariant case the exponents in the expansion are less than or equal to the dimension of the fixed point set. We deduce a corresponding equivariant arithmetic Hilbert-Samuel theorem in Arakelov geometry.
\end{abstract}
\begin{center}
\end{center}
\thispagestyle{empty}
\setcounter{page}{1}

\section{Introduction}

Given a holomorphic vector bundle $E$ equipped with an Hermitian metric $h^E$ on a compact Hermitian manifold, Ray and Singer constructed the holomorphic torsion $T(M,E)\in\BR$ as a regularized determinant of Kodaira-Laplace operators. This definition was  extended to an equivariant setting (\cite{K1}) and to orbifolds by Ma (\cite{Ma1}). For powers $L^p$ of a positive Hermitian line bundle $(L,h^L)$, Bismut and Vasserot (\cite{BV}) determined the first terms of an asymptotic expansion of $T(M,E\otimes L^p)$ as $p\to+\infty$. Other extensions and applications to physics of asymptotics of holomorphic torsion have been given by Berman (\cite{Ber}), Su (\cite{Su}), Larra\'in-Hubach (\cite{La}) and, for torsion forms, by Puchol (\cite{Puchol}).
Finski (\cite{Fi}) proved in the orbifold situation that this expansion has two kinds of terms, of the form $p^\ell$ and $p^\ell\log p$ for $-\infty<\ell\leq\dim M$. Also he showed in \cite[Th.\ 1.5]{Fi} that the coefficient of the $p^\ell\log p$-terms had to be independent of the metrics on $M$, $E$ and $L$. The first result of this article is the computation of the $p^\ell\log p$-part of the asymptotic expansion in terms of explicitly given characteristic classes.

Let $\BZ/(m)$ denote the cyclic group of order $m\in\BZ^+$. For each $m$ we fix a primitive $m$-th root of unity $\zeta_m$. For a compact effective orbifold $M$ let $M\dot\cup \Sigma M=\dot\bigcup_{j\in J} M_j$ denote the decomposition of $M$ and its strata $\Sigma M$ into connected components and let $m_j$ be the multiplicity of $M_j$.
\begin{theor}\label{AsympOrbi}
Let $M$ be a compact effective Hermitian orbifold of dimension $n$ and let $(L,h^L)$, $(E,h^E)\to M$ be proper Hermitian orbifold vector bundles with $\rk L=1$ and $L$ positive. Let $\nm:=\max_{M_j\subset M_{\rm sing}}M_j$ denote the maximum of the dimensions of singular strata.  Then there is a number $m\in\BZ^+$ and families $\tilde\beta_\bullet:\BN_0\to\BR$, $\tilde\kappa_\bullet:\BZ/(m)\x\BN_0\to\BR$ such that for any $k\in\BN_0$ the orbifold holomorphic torsion satisfies for $p\to+\infty$
\begin{eqnarray*}
T(M,E\otimes L^p)&=&\log p\cdot\sum_j\frac1{m_j}\int_{M_j}\left(n\cdot \Td^\Sigma(TM)-\Td^{'\Sigma}(TM)\right)\ch^\Sigma(E\otimes L^p)
\\&&+\sum_{j=0}^kp^{n-j}\tilde\beta_j+\sum_{j=0}^{k+\nm-n}\sum_{r\in\BZ/(m)}p^{\nm-j}\zeta_m^{r p}\tilde\kappa_{r,j}+o(p^{n-k}).
\end{eqnarray*}
\end{theor}
In particular it turns out that $p^\ell\log p$-part equals $\log p$ times a polynomial in $p$. In the special case where $M$ is a manifold the theorem shows that the coefficient of $\log p$ of the asymptotic expansion of the classical non-orbifold holomorphic torsion is given by
$$
\int_{M}(n\cdot  \Td(TM)-\Td'(TM))\ch(E\otimes L^p).
$$
From Th.\ \ref{AsympOrbi}, we deduce the analogous result for the equivariant torsion $T_g(M,E\otimes L^p)$ in the case of an action of finite order of an isometry $g$ on $M, E$ and $L$. In the equivariant case, however, the largest exponent of $p$ is given by the maximum $\nm$ of the dimensions of fixed point submanifolds instead of by $\dim M$.
\begin{theor}\label{AsympEquiv}
Let $M$ be a compact Hermitian manifold of dimension $n$ and let $(L,h^L),(E,h^E)\to M$ be Hermitian vector bundles with $\rk L=1$ and $L$ positive. Assume that $M$, $E$ and $L$ are invariant under the action of an holomorphic isometry $g$ of finite order $m$ and that $E$ and $L$ are equipped with a $g$-equivariant structure. Let $\nm$ denote the largest dimension of the components of the fixed point submanifold $M^g$. Then there is a family $\kappa_\bullet':\BZ/(m)\x\BN_0\to\BR[\zeta_m]$ such that for any $k\in\BN_0$ the equivariant holomorphic torsion satisfies for $p\to+\infty$
\begin{eqnarray*}
T_g(M,E\otimes L^p)
&=&\log p\cdot\int_{M^g}(n\cdot \Td_g(TM)-\Td_g'(TM))\ch_g(E\otimes L^p)
\\&&+\sum_{j=0}^kp^{\nm-j}\sum_{r\in\BZ/(m)}\zeta_m^{r p}\kappa_{r,j}'+o(p^{\nm-k}).
\end{eqnarray*}
\end{theor}
This formula has been proven in the case $M$ symmetric, $E$ trivial, $g\in G$ arbitrary by \cite[Th.\ 5.5]{K4} and for $M$ homogeneous, $\dim M^g=0$ by \cite[Th.\ 8.1]{K4}. Also, the main result \cite[Th.\ 1]{T} of an as of yet unpublished preprint by Te\ss mer states that for $p\to+\infty$ the part of differential form degree $2k$ of equivariant torsion forms is given by a linear combination of $p^{\nm+k}$, $p^{\nm+k}\log p$ and coefficients as in Th.\ \ref{AsympEquiv} up to $o(p^{\nm+k})$.

We also deduce corresponding anomaly formulae for the asymptotic expansion of $L^2$-metrics on the Knudsen-Mumford determinant of cohomology in Cor.\ \ref{BV-L2-Orbi} and Cor.\ \ref{BV-L2-equiv}.

The main application of holomorphic torsion is the construction of the determinant of a direct image of Hermitian vector bundles within the context of Arakelov geometry. Similarly to the role played by the Kodaira vanishing theorem in algebraic geometry, describing the asymptotic behaviour of holomorphic torsion for increasing tensor powers of positive line bundles $L\to M$ yields arithmetic Hilbert-Samuel theorems in Arakelov geometry. The first result of this kind has been given by Gillet and Soul\'e in \cite[Th.\ 8]{GS8} using Bismut-Vasserot's theorem.
Since then, numerous articles have been published that extend and apply this result, in particular by Abbes and Bouche \cite{ABo}, S.-W. Zhang \cite{Z}, Randriambololona \cite{R}, Yuan \cite{Y}, Moriwaki \cite{Mo}, Berman and Freixas i Montplet \cite{BeFr}, Botero, Burgos Gil, Holmes and de Jong \cite{BoBuHoJo} and Chen and Moriwaki \cite{CMo}.

We state the result for an equivariant arithmetic variety $X$, i.e. a regular integral scheme which is endowed with a $\mn$-projective action, projective and flat over $\Spec\BZ$. Let $\ar{\rm deg}(V,h^V)\in\BR$ for a finitely generated Hermitian $\BZ$-module $(V,h^V)$ denote Gillet-Soul\'e's arithmetic degree. We equip $H^\bullet(X(\BC),E_\BC\otimes L_\BC^p)$ with the $L^2$-metric induced by the Hodge embedding.
\begin{theor}\label{eqHilbert-Samuel}
Let $f:X\to\Spec\BZ$ be a flat projective arithmetic variety of dimension $n+1$, equipped with a K\"ahler metric $g^{TX}$ over $X(\BC)$. Let $(L,h^L),(E,h^E)\to X$ be Hermitian vector bundles with $\rk L=1$ and $L$ positive in the sense that $H^q(X,E\otimes L^p)=0$ for $q>0$ and $p>>0$ and that $c_1(L_\BC,h^L)>0$. Assume that $X$, $E$ and $L$ carry a $\mu_m$ action for $m\in\BZ^+$ and that the metrics are invariant under this action. Let $\nm+1$ denote the largest dimension of the components of the fixed point scheme $X^{\mu_m}$ and let $N$ denote the normal bundle for the embedding $X^\mn\hookrightarrow X$. Then there is a family $\hat\kappa_\bullet:\BZ/(m)\x\BN_0\to\BR[\zeta_m]$ such that for $p\to+\infty$ the following identity of complex numbers holds:
\begin{eqnarray*}
\lefteqn{
-\sum_{k\in{\BZ}/(m)}{\zeta_m^k}\cdot
\log\covol\,({H^0(X,E\otimes L^p)}_k,|\cdot|_{L^2})
}\hspace{2cm}\\&=&
\ar\deg f^\mn_*\bigg(\frac{p^{n-\rk N+1}}{(n-\rk N+1)!}\frac{\sum_{\ell\in\BZ/(m)}\zeta_m^{\ell}\rk E_{\ell}}{\prod_{r\in\BZ/(m)}\left(1-\zeta_m^{-r}\right)^{\rk N_r}}
\sum_{k\in\BZ/(m)}\zeta_m^{pk}\ar c_1(L_k,h^L)^{n-\rk N+1}
\bigg)
\\&&+\log p\cdot\int_{X^{\mu_m}(\BC)}(n\cdot \Td_g(TX_\BC)-\Td_g'(TX_\BC))\ch_g(E_\BC\otimes L_\BC^p)
\\&&+\sum_{j=0}^kp^{\nm-j}\sum_{r\in\BZ/(m)}\zeta_m^{r p}\hat\kappa_{r,j}+o(p^{\nm-k}),
\end{eqnarray*}
where the coefficients $\hat\kappa_{r,j}$ for $j>\nm$ depend only on the complex points and the metrics.
\end{theor}
Thus a certain weighted logarithm of the covolume of the sections is given by a weighted height of the fixed point scheme, a topological term and and an asymptotic expansion in powers of $p$, where the largest exponent of $p$ is smaller than the dimension $\nm+1$ of the fixed point scheme. The analogue result for arithmetic varieties over general regular arithmetic rings $D$ can be obtained by considering the result above for all complex embedding of $D$, as our contribution concerns the complex points. 
In contrast in previous versions of an arithmetic Hilbert-Samuel theorem the largest exponent is given by $\dim X$, and only the highest  degree term of the $\log p$-summand was known.

{\bf Acknowledgements.} The author is grateful to Siarhei Finski and Xiaonan Ma for valuable discussions and comments.

\section{Effective orbifolds}
Throughout this article, a quotient $\frac{g'}{g}$ of two bilinear forms with $g$ non-degenerate is understood to be the endomorphism $A$ such that $g(A\cdot,\cdot)=g'$. We shall apply metrics $h^E$ on bundles $E$ to subbundles and to bundles constructed from $E$ without changing the notation for the metric. For an element $\omega$ of a graded algebra we denote by $\omega^{[k]}$ the summand in degree $k$.

We consider the same (effective) orbifold setting as in \cite{Kaw1}, \cite[Section 1.1]{Ma1} and \cite[section 5]{Fi}. Let ${\cal M}_s$ denote the category such that the objects are pairs $(G,M)$, where $G$ is a finite group acting effectively on a smooth connected manifold $M$. The morphisms $\Phi:(G,M)\to(G',M')$ shall be families of open embeddings $\vp:M\to M'$ such that the following properties hold:
\begin{enumerate}
\item For every $\vp\in\Phi$ there exists an group monomorphism $\lambda_\vp:G\to G'$ such that $\vp$ is $\lambda_\vp$-equivariant.
\item For all $\phi\in\Phi$, $g'\in G'$ satisfying $g'\cdot(\vp(M))\cap \vp(M)\neq\emptyset$, one finds $g'\in\lambda_\vp(G)$.
\item Any $\vp\in\Phi$ satisfies $\Phi=\{g'\vp|g'\in G'\}$.
\end{enumerate}
\begin{defin}
Consider a paracompact Hausdorff space $M$. Let $\cal U$ be a covering of $M$ by connected open subsets which is \emph{dense}, i.e. for all $U,U'\in\cal U$ and $x\in U\cap U'$ there is a $U''\in\cal U$ such that $x\in U''\subset U\cap U'$.
An \emph{orbifold atlas} $\cal V$ on $M$ is a family of ramified coverings ${\cal V}(U):(G_U,\tilde U)\to U$ for $U\in\cal U$ (the \emph{orbifold charts}) such that $(G_U,\tilde U)\in{\cal M}_s$ and:
\begin{enumerate}
\item For any $U\in\cal U$ one has $U\cong \tilde U/G_U$.
\item For any $U,V\in\cal U$, $U\subset V$ there is a morphism $\vp_{VU}:(G_U,\tilde U)\to(G_V,\tilde V)$ that covers the inclusion $U\subset V$.
\item For any $U,V,W\in\cal U$, $U\subset V\subset W$ the equality $\vp_{WU}=\vp_{WV}\circ\vp_{VU}$ holds.
\end{enumerate}
A pair $(M,{\cal V})$ of $M$ together with an orbifold atlas is called an \emph{(effective) orbifold}. 
\end{defin}
Let $\#S$ denote the cardinality of a set $S$.
\begin{defin}
Consider an orbifold $M$. For $x\in M$ choose a chart $(G_x,\tilde U_x)\to U_x$ such that $x\in\tilde U_x$ is fix under the action of $G_x$. By the above definition $G_x$ is uniquely determined up to isomorphism. The \emph{singular points} of $M$ are $M_{\rm sing}:=\{x\in M|\#G_x>1\}$, the \emph{regular points} are $M_{\rm reg}:=M\setminus M_{\rm sing}$.
\end{defin}
\begin{defin}
An \emph{orbifold vector bundle} $\pi:E\to M$ is an orbifold $E$ together with $G_U,G_U^E$-equivariant local vector bundles $\tilde\pi_U:\tilde E_U\to\tilde U$ for local charts $(G^E_U,\tilde E_U)$, $(G_U,\tilde U)$ of $E$ and $M$, resp., such that the transition maps are equivariant, where $G_U=G^E_U/K^E_U$ with  $K^E_U:=\ker(G^E_U\to{\rm Diffeo}(\tilde U))$.
The vector bundle $E$ is called \emph{proper} if $\#K^E_U=1$.

Let $\widetilde{E^{\rm pr}_U}$ denote the maximal $K^E_U$-invariant subbundle of $\tilde E_U$. Let $E^{\rm pr}$ be the proper orbifold bundle given locally by $(G_U,\widetilde{E^{\rm pr}_U})$.
\end{defin}
Riemannian, complex, orientation structures etc.\ are defined by replacing the category $\cal M$ correspondingly.

A holomorphic (or smooth) section $s:M\to E$ is a map covered locally by holomorphic (or smooth) $G^E_U$-invariant sections $\tilde s_U:\tilde U\to\tilde E_U$. For a form $\alpha$ on $M$ with ${\rm supp}\,\alpha\subset U$ set $\int_M\alpha:=\frac1{\#G_U}\int_{\tilde U}\tilde \alpha_U$. We extend this definition to forms with compact support by linearity.

By \cite[Lemma 5.4.3]{MaM}, \cite[p. 76]{Kaw1}, for any $x\in M$ there is a local chart $\tilde U_x\subset\BR^n\to U_x$ such that $G_x$ acts linearly on $\tilde U_x$. Let $\{[1],[h^1_x],\dots,[h^{\rho_x}_x]\}$ denote the set of conjugacy classes in $G_x$ with representatives $h^j_x\in G_x$. Let $C_{G_x}(h^j_x)$ denote the centralizer of each $h^j_x$ and let $\tilde U^{h^j_x}_x\subset\tilde U_x$ denote the fixed point set of $h^j_x$. Then there is a bijection (\cite[p. 77]{Kaw1}) to a disjoint union of orbifolds
$$
\Big\{
(y,[h^j_y])\Big|y\in U_x,j=1,\dots,\rho_y
\Big\}\stackrel{\cong}\to\dot\bigcup_{j=1}^{\rho_x}\tilde U^{h^j_x}/C_{G_x}(h^j_x).
$$
The \emph{strata of the orbifold} $M$ are the elements of the set
$$
\Sigma M:=\Big\{
(x,[h^j_x])\Big|x\in M,j=1,\dots,\rho_x,x\in M_{\rm sing}
\Big\}.
$$
Set $K^j_x:=\ker\left(C_{G_x}(h^j_x)\to{\rm Diffeo}(U^{h^j_x}_x)\right)$. $\Sigma M$ can be equipped with the orbifold atlas
$$
\Big\{
(C_{G_x}(h^j_x)/K^j_x,\tilde U^{h^j_x}_x)\to \tilde U^{h^j_x}_x/C_{G_x}(h^j_x)
\Big|x\in M,j=1,\dots,\rho_x,x\in M_{\rm sing},U_x\mbox{ chart}
\Big\}.
$$
Let $\dot\bigcup_{j\in J} M_j$ denote the decomposition of $M\dot\cup \Sigma M$ into connected components and set $n_j:=\dim_\BC M_j$. The \emph{multiplicity of} $M_j$ is given by $m_j:=\#K^r_x$ for any $(x,[h^r_x])\in M_j\subset\Sigma M$ and $m_j:=1$ for $M_j\subset M$. This is well-defined, as this integer is locally constant in $x$ (\cite[p. 77]{Kaw1}).

Let ${\mathfrak A}^{p,q}(M,E)$ denote the forms of holomorphic degree $p$ and antiholomorphic degree $q$ with coefficients in $E$ and set ${\mathfrak A}^{p,q}(M):={\mathfrak A}^{p,q}(M,{\cal O})$.
Consider a holomorphic Hermitian vector bundle $E\to M$ over a complex manifold $M$. Let $\nabla^E$ be the associated canonical covariant derivative with curvature $\Omega^E\in{\mathfrak A}^{1,1}(M,\End\, E)$. 
Let $g$ be a biholomorphic map acting on $M$. Assume that $E$ is $g$-invariant as a holomorphic Hermitian bundle and that $E$ is equipped with an equivariant structure $g^E$. The
Hermitian vector bundle $\mtr E$ splits on the fixed point submanifold $M^g$ into a direct sum
$E_{|M^g}=\bigoplus_{\theta\in\BR/(2\pi\BZ)}\mtr E_{\theta}$, where the equivariant structure $g^E$ of $E$
acts on
$\mtr E_\theta$ as $e^{i\theta}$.
\begin{defin}
The \emph{Chern character form} on $M^g$ is defined as
\begin{eqnarray*}
\ch_g(\mtr E,h^E)&:=&\sum_\theta e^{i\theta}\ch(\mtr E_\theta,h^{E_\theta})
:=\sum_\theta e^{i\theta}\ch\left(\frac{-\Omega^{\mtr E_\theta}}{2\pi i}\right)
\\ &=&\Tr g^E+\sum_\theta e^{i\theta}
c_1(\mtr E_\theta,h^{E_\theta})+\sum_\theta e^{i\theta} \left(\frac{1}{2}c_1^2(\mtr E_\theta,h^{E_\theta}) -c_2(\mtr
E_\theta,h^{E_\theta})\right)+\dots
.
\end{eqnarray*}
The \emph{Todd form} of an equivariant vector bundle is defined as
\begin{eqnarray*}
\Td_g(\mtr E,h^E)&:=&\frac{c_{{\rm rk}\,E_0}(\mtr E_0,h^{E_0})}
{\ch_g(\sum_{j=0}^{{\rm rk}\,E} (-1)^j \Lambda^j \mtr E^*,h^{E})}
=\Td\left(\frac{-\Omega^{E_0}}{2\pi i}\right)
\prod_{\theta\notin2\pi\BZ}\frac\Td{c_{\rm top}}\left(\frac{-\Omega^{E_\theta}}{2\pi i}+i\theta\right).
\end{eqnarray*}
Set as in \cite[Def. 2.4]{Bimm}
$$
\Td'_t(\mtr E,h^E):=\frac\partial{\partial \vep}_{|\vep=0}
\Td\Big(\frac{-\Omega^{E_0}}{2\pi i}+\vep{\rm id}_{E_0}\Big)\prod_{\theta\notin2\pi\BZ} \frac{\Td}{c_{\rm top}}\Big(\frac{-\Omega^{E_\theta}}{2\pi i}+i\theta+\vep{\rm id}_{E_\theta}\Big).
$$
These are forms in $\bigoplus_{p=0}^{\dim M}{\mathfrak A}^{p,p}(M)$.
\end{defin}
The associated characteristic classes are denoted by $\ch_g(E),\Td_g(E)$ and $\Td'_g(E)$. 
For
$\zeta\in S^1$ and $s>1$ consider the zeta function
$$ L(\zeta,s)=\sum_{k=1}^\infty \frac{\zeta^k}{k^s}
$$ and its meromorphic continuation to $s\in\BC$. 
Consider the formal power series in $x$
$$
\widetilde R(\zeta,x):=\sum_{n=0}^\infty 
\left(\frac{\partial L}{\partial s}(\zeta,-n)+
L(\zeta,-n)\sum_{j=1}^n\frac{1}{2j}\right)\frac{x^n}{n!}.
$$
\begin{defin} The \emph{Bismut equivariant $R$-class} of an equivariant holomorphic
vector bundle $E$ is
defined as the cohomology class
$$
R_g(E):=\left[
\sum_{\theta}\left(\Tr \widetilde R(e^{i\theta},-\frac{\Omega^{
E_\theta}}{2\pi i})-\Tr \widetilde R(e^{-i\theta},\frac{\Omega^{E_\theta}}{2\pi
i})\right)
\right].
$$
\label{BisRg}
\end{defin}
This generalizes the non-equivariant \emph{Gillet-Soul\'e $R$-class}. 
Given a characteristic class $\Phi$, let
$$\widetilde\Phi(\xi)\in \widetilde{\mathfrak A}(M):=\bigoplus_{p\geq 0}{\mathfrak A}^{p,p}(M)/({\rm im}\,\partial +{\rm im}\,\overline{\partial})$$
be the associated equivariant Bott-Chern secondary class as introduced in 
\cite[Th.\ 3.4]{lrr1}, following the non-equivariant axiomatic definition in \cite[Section f]{BGS1}. Notice though that there are two sign conventions for Bott-Chern classes, one used by Gillet and Soul\'e and one used by Bismut, and in this article we work with Bismut's convention as used in \cite{BGS1}: For any short exact sequence of $g$-equivariant holomorphic vector bundles $\xi:0\to E'\to E\to E''\to0$, equipped with any $g$-invariant Hermitian metrics $h^{E'},h^E,h^{E''}$, the associates Chern-Weil differential forms satisfy
$$
\frac{\bar\partial\partial}{2\pi i}\tilde\Phi(\xi)=\Phi(E,h^E)-\Phi(E'\oplus E'',(h^{E'},h^{E''})).
$$
In particular, for two Hermitian metrics $h^E,h^{'E}$ on a holomorphic vector bundle $E$ the class $\widetilde\Phi$ provides the transgression
$$\frac{\bar\partial\partial}{2\pi i}\widetilde\Phi_g(E,h^E,h^{'E})= \Phi_g(E,h^{'E})-\Phi_g(E,h^{E}).$$
For example, $\widetilde{c_1}(E,h^E,h^{'E})=\log\det\frac{h^E}{h^{'E}}$.

Following Kawasaki (\cite[p.\ 153]{Kaw2}), Ma defines the following characteristic classes for proper orbifold vector bundle $E\to M$ on an orbifold $M$ in \cite[(1.7)]{Ma1}: Let $E$ be represented by local charts $\tilde E_U\to \tilde U$ and consider the characteristic forms above as forms on $\tilde U^g$. Then the orbifold characteristic classes are represented locally by the corresponding $g$-equivariant characteristic forms on each connected component of $\tilde U^g/C_{G_U}(g)$. Hence they are given by
\begin{eqnarray*}
\ch^\Sigma(E)_{|\tilde U^g/C_{G_U}(g)}:=\ch_g(\tilde E_U),\ \Td^\Sigma(E)_{|\tilde U^g/C_{G_U}(g)}:=\Td_g(\tilde E_U),\ \Td^{'\Sigma}(E)_{|\tilde U^g/C_{G_U}(g)}:=\Td_g'(\tilde E_U).
\end{eqnarray*}
Analogously, orbifold versions of Bott-Chern classes are defined (\cite[Eq.\ (1.8)]{Ma1}).

\section{Holomorphic torsion}
Let $M$ be a compact complex orbifold of dimension $n$ equipped with an Hermitian metric with associated K\"ahler form $\omega$. Consider an Hermitian holomorphic orbifold bundle $(E,h^E)\to M$ and the corresponding Dolbeault operator $\bar\partial^E:{\mathfrak A}^{0,q}(M,E)\to {\mathfrak A}^{0,q+1}(M,E)$ acting on antiholomorphic $q$-forms with coefficients in $E$. The space ${\mathfrak A}^{0,q}(M,E)$ is equipped with the $L^2$-metric
$$
(\eta,\eta')_{L^2}:=\int_M\langle\eta_{|x},\eta'_{|x}\rangle\frac{\omega^{^n}}{(2\pi)^nn!}.
$$
In \cite[Section 2.2]{Ma1} Ma extends Ray-Singer's definition of holomorphic torsion to orbifolds: Using the formal adjoint $\bar\partial^{E*}$ with respect to this $L^2$-metric, let $\square^q:=(\bar\partial^{E}+\bar\partial^{E*})^2$ denote the Kodaira-Laplace operator acting on ${\mathfrak A}^{0,q}(M,E)$. By \cite[p. 2214]{Ma1}, $\square^q$ is a self-adjoint elliptic operator and thus the spectrum $\sigma(\square^q)$ on the compactum $M$ is real and discrete. The zeta function
$$
Z^\Sigma(s):=\sum_{q=0}^n(-1)^{q+1}q\sum_{\substack{\lambda\in\sigma(\square^q)\\ \lambda\neq0}}\frac{\dim{\rm Eig}_\lambda(\square^q)}{\lambda^{s}}
$$
is well-defined for ${\rm Re}\,s>>0$ and its meromorphic continuation to the complex plane has a well-defined derivative $T(M,E):=(Z^{\Sigma})'(0)\in\BR$ at $s=0$, the \emph{orbifold holomorphic torsion}.

Now consider an Hermitian holomorphic bundle $(E,h^E)\to M$ over an Hermitian manifold and assume that a holomorphic isometry $g$ acts on $M$, leaves $E$ invariant and that $E$ is equipped with an equivariant structure $g^E$. This induces an action of $g$ on the eigenspaces of $\square^q$ and we define
$$
Z_g(s):=\sum_{q=0}^n(-1)^{q+1}q\sum_{\substack{\lambda\in\sigma(\square^q)\\ \lambda\neq0}}\frac{\Tr g_{|{\rm Eig}_\lambda(\square^q)}}{\lambda^{s}}
$$
for ${\rm Re}\,s>>0$ and the \emph{equivariant holomorphic torsion} as $T_g(M,E):=Z_g'(0)\in\BC$ as in \cite{K1}.
\section{The asymptotic expansion}
Let $E,L, F$ be proper Hermitian holomorphic orbifold bundles over a compact effective Hermitian orbifold $M$. Assume that $F$ is flat and that $L$ is a positive line bundle.
Let $\square_p^q$ denote the Kodaira-Laplace operator acting on antiholomorphic forms of degree $q$ with coefficients in $E\otimes L^p\otimes F$. Let $Z_p$ denote the zeta function defining the torsion $T(M,E\otimes L^p\otimes F)=Z'_p(0)$ of $E\otimes L^p\otimes F$. For $M$ smooth Bismut and Vasserot prove in \cite[Eq.\ (14)]{BV} that there is an asymptotic expansion
\begin{equation}
p^{-n}\Tr\exp\left(
-\frac{t}{p}\square_p^q
\right)=\rk F\cdot\sum_{j=-n}^ka^q_{p,j}t^j+o(t^k),
\end{equation}
where the $a^q_{p,j}$ do not depend on $F$. Using this equation (\cite[Eq.\ (2.19)]{Fi}) and $b_{p,j}:=\sum_{q=0}^n(-1)^qqa^q_{p,j}$, Finski shows in \cite[Th.\ 2.7]{Fi} that
\begin{equation}\label{FinskiSmooth}
Z_p'(0)=\sum_{j=0}^kp^{n-j}(\alpha_j\log p+\beta_j)+o(p^{n-k})
\end{equation}
where $\alpha_j,\beta_j$ are given in terms of $b_{p,j}$ such that they take the form $\rk F$ times a factor which is independent of $F$.
For an orbifold $M$, \cite[Th.\ 5.12]{Fi} extends this result using coefficients $\tilde\alpha_j,\tilde\beta_j$:
\begin{theor}\label{FinskiMain}
(Finski, \cite[Th.\ 1.5]{Fi})
Let $M$ be a compact effective Hermitian orbifold of dimension $n$ and let $(L,h^L),(E,h^E)\to M$ be proper Hermitian orbifold vector bundles with $\rk L=1$ and $L$ positive. Then
\begin{eqnarray*}\label{FinskiMainFml}
\nonumber
\lefteqn{
T(M,E\otimes L^p)
=\sum_{j=0}^kp^{n-j}(\tilde\alpha_{j}(\rho)\log p+\tilde\beta_{j}(\rho))
}\\&&+\sum_r\sum_{j=0}^{k+n_r-n}\frac{p^{n_r-j}}{m_r}e^{ip\theta_r}(\gamma_{r,j}(\rho)\log p+\kappa_{r,j}(\rho))
+o(p^{n-k})
\end{eqnarray*}
where $e^{i\theta_r}$ is the action on $L$ on the stratum $M_r$ (\cite[(5.5)]{Fi}) and $\gamma_{r,j},\kappa_{r,j}$ are integrals over local data (\cite[Th.\ 5.12]{Fi}).
\end{theor}
As is detailed in \cite[Rem. 1.7(3)]{Fi}, \cite[Th.\ 5.12]{Fi} shows that the $\tilde\alpha_j,\tilde\beta_j$ are given by integrals over the same local quantities as for $\alpha_j,\beta_j$. 

\begin{proof}[Proof of Th.\ \ref{AsympOrbi}]
Let $Z_p$ denote the zeta function defining the torsion $T(M,E\otimes L^p)=Z'_p(0)$ of $E\otimes L^p$ and set $\tilde Z_p(s):=p^{s-n}Z_p(s)$. Hence one gets the analogue of \cite[Eq.\ (40)]{BV}
$$
Z_p'(0)=-p^n\log p\cdot\tilde Z_p(0)+p^n\tilde Z'_p(0).
$$
Choose $m$ as the lcm of the orders of local groups $G_U$ of a finite atlas. \cite[Th.\ 5.18]{Fi} proves that $\tilde Z'_p(0)$ has an asymptotic expansion of the form
$$\sum_{j=0}^kp^{-j}\tilde\beta_j+\sum_{j=0}^{k+\nm-n}\sum_{r\in\BZ/(m)}p^{\nm-n-j}\zeta_m^{r p}\tilde\kappa_{r,j}+o(p^{-k})$$
as in the statement.
By definition, using the number operator $N:{\mathfrak A}^{0,q}(M,E)\to {\mathfrak A}^{0,q}(M,E)$, $\a\to q\a$ for $q\in\BN_0$,
$$\tilde Z_p(s)=\frac{-p^{-n}}{\Gamma(s)}\int_0^{+\infty}t^{s-1}\Tr_sNe^{-\frac{t}{p}\square_p}\,dt.$$
This integral is well-defined for $\Re s>>0$, as $\square_p^q$ is invertible for large $p$ and $q\geq1$ and the summand with $q=0$ is erased by $N$.
Now equip $M$ with the K\"ahler metric induced by $c_1(L)$. For K\"ahler metrics \cite[Th.\ 2.4]{Ma1} states that
$$
\Tr_sNe^{-u\square_p}=\frac1uC_{-1}+C_0+O(u)
$$
as $u\searrow0$ where
\begin{equation*}
C_0=\sum_j\frac1{m_j}\int_{M_j}\left(\dim M\cdot \Td^\Sigma(TM)-\Td^{'\Sigma}(TM)\right)\ch^\Sigma(E\otimes L^p)
\end{equation*}
by \cite[Eq.\ (2.24)]{Ma1}. Replacing $u$ by $\frac{t}{p}$ implies $\tilde Z_p(0)=-p^{-n}C_0$.
Finally \cite[p.\ 44]{Fi} shows that $\tilde Z_p(0)$ is independent of the metric on $M$, using Ma's analogue \cite[Th 0.1]{Ma1} of the anomaly formula for orbifolds and \cite[Th.\ 1.2]{DLM2}. Thus the formula holds for arbitrary Hermitian metrics.
\end{proof}
Now we shall check that Th.\ \ref{AsympOrbi} is compatible with the explicit formulae for the first coefficients given by Bismut and Vasserot in \cite[Th.\ 8]{BV} and Finski in \cite[Th.\ 1.3, Th.\ 1.5, Rem. 1.7(3)]{Fi}. In the case $E$ trivial and $M$ smooth this verification has already been done in \cite[Th.\ 6.5]{K4}. We shall use the expansions
\begin{eqnarray*}
\Td(TM)&=&1+\frac{c_1(TM)}2+{\rm higher\ order},
\\ \frac{\Td'(TM)}{\Td(TM)}-n&=&-\frac{n}2-\frac{c_1(TM)}{12}+{\rm higher\ order},
\\ \ch(L^p)&=&\frac{p^n}{n!}c_1(L)^n+\frac{p^{n-1}}{(n-1)!}c_1(L)^{n-1}+O(p^{n-2}).
\end{eqnarray*}
Finski's formula for $\gamma_{r,0}$ in \cite[Eq.\ (1.9)]{Fi} follows immediately from the summands with $m_j>1$ in Th.\ \ref{AsympOrbi}.
If the singular strata have codimension at least 2, the coefficient of $\log p$ in Th.\ \ref{AsympOrbi} equals
\begin{eqnarray*}
\lefteqn{
\log p\cdot\int_{M_{\rm reg}}\left(n\cdot \Td(TM)-\Td'(TM)\right)\ch(E\otimes L^p)
+o(p^{n-1})
}\hspace{2cm}\\&=&\log p\cdot\int_{M_{\rm reg}}
\left(1+\frac{c_1(TM)}2\right)
\left(\frac{n}2+\frac{c_1(TM)}{12}\right)
\left(\rk E+c_1(E)\right)
\\&&\cdot\left(
\frac{p^n}{n!}c_1(L)^n+\frac{p^{n-1}}{(n-1)!}c_1(L)^{n-1}
\right)+o(p^{n-1})
\\&=&\log p\cdot\int_{M_{\rm reg}}\bigg(
\frac{p^n}{2(n-1)!}c_1(L)^n\rk E
\\&&+\Big(
\frac{3n+1}{12}
c_1(TM)\rk E
+\frac{n}2c_1(E)\Big)
\cdot\frac{p^{n-1}}{(n-1)!}c_1(L)^{n-1}
\bigg)
+o(p^{n-1}).
\end{eqnarray*}
This corresponds exactly to the results by Bismut,Vasserot and Finski.

Let $(g^{TM},h^E,h^L)$, $(g^{'TM},h^{'E},h^{'L})$ be two triples of metrics on $TM, E$ and $L$. Set $\lambda_p:=\det H^0(M,E\otimes L^p)^{-1}$.
The Hodge embedding $H^0(M,E\otimes L^p)\cong\ker\square^0\subset \Gamma^\infty(M,E\otimes L^p)$ induces associated $L^2$-metrics $h'_{\lambda,p},h_{\lambda,p}$ on $\lambda_p$.
\begin{cor}\label{BV-L2-Orbi}
Let $g^{TM},g^{'TM}$ be two K\"ahler metrics on a compact $n$-dimensional effective orbifold $M$ and let $(L,h^L),(E,h^E)\to M$ be proper Hermitian orbifold vector bundles with $\rk L=1$ and $L$ positive. Let $h^{'E}$ denote a second Hermitian metric on $E$. Then there is a number $m\in\BZ^+$ and families $\beta''_\bullet:\BN_0\to\BR$, $\kappa''_\bullet:\BZ/(m)\x\BN_0\to\BR$ such that for any $k\in\BN_0$ the $L^2$-metrics on $\lambda_p$ satisfy for $p\to+\infty$
\begin{eqnarray*}
\log\frac{h'_{\lambda,p}}{h_{\lambda,p}}&=&p^n\int_M\left(\rk E\cdot \log\det\frac{g^{TM}}{g^{'TM}}
+\log\det\frac{h^{E}}{h^{'E}}
\right)\cdot\frac{c_1(L,h^L)^n}{n!}
\\&&+\sum_{j=1}^kp^{n-j}\beta''_j+\sum_{j=0}^{k+\nm-n}\sum_{r\in\BZ/(m)}p^{\nm-j}\zeta_m^{r p}\kappa''_{r,j}+o(p^{n-k}).
\end{eqnarray*}
\end{cor}
This extends the asymptotic anomaly formula \cite[Th.\ 10]{BV} to the K\"ahler orbifold case. Varying $h^L$ adds a term that is proportional to $p^{n+1}$.
\begin{proof}
Ma shows in \cite[Th.\ 0.1]{Ma1} the following anomaly formula for two pairs of metrics $(h^E,h^L)$, $(h^{'E},h^{'L})$ on $E$ and $L$ and two K\"ahler metrics $g^{TM}$, $g^{'TM}$ on $M$:
\begin{eqnarray*}
\lefteqn{
\log\frac{h'_{\lambda,p}}{h_{\lambda,p}}-T(M,E\otimes L^p,(g^{'TM},h^{'E},h^{'L}))+T(M,E\otimes L^p,(g^{TM},h^E,h^L))
}\hspace{2cm}
\\&=&\sum_j\frac1{m_j}\int_{M_j}\widetilde\Td^\Sigma(TM,g^{TM},g^{'TM})\ch^\Sigma(E\otimes L^p,(h^E,h^L))
\\&&+\sum_j\frac1{m_j}\int_{M_j}\Td^\Sigma(TM,g^{'TM})\widetilde\ch^\Sigma(E\otimes L^p,(h^E,h^L),(h^{'E},h^{'L})).
\\&\stackrel{h^L=h^{'L}}=&p^n\int_{M_{\rm reg}}\left(\frac12\widetilde{c_1}(TM,g^{TM},g^{'TM})\rk E+\widetilde{c_1}(E,h^{E},h^{'E})\right)\frac{c_1(L,h^L)^n}{n!}
\\&&+\sum_{j=1}^np^{n-j}\tilde\beta''_j+\sum_{j=0}^\nm\sum_{r\in\BZ/(m)}p^{\nm-j}\zeta_m^{r p}\tilde\kappa''_{r,j}
\\&=&p^n\int_{M_{\rm reg}}\left(\frac12\log\det\frac{g^{TM}}{g^{'TM}}\rk E+\log\det\frac{h^{E}}{h^{'E}}\right)\frac{c_1(L,h^L)^n}{n!}
\\&&+\sum_{j=1}^np^{n-j}\tilde\beta''_j+\sum_{j=0}^\nm\sum_{r\in\BZ/(m)}p^{\nm-j}\zeta_m^{r p}\tilde\kappa''_{r,j}.
\end{eqnarray*}
The right hand side is a polynomial in $p$ of degree $n$ with coefficients in functions of the form $e^{i\theta p}$. As the coefficient of $\log p$ in Th.\ \ref{AsympOrbi} is independent of the metrics, the difference of the holomorphic torsions has an asymptotic expansion in powers of $p$ as in the statement. 
In \cite[Th.\ 1.5]{Fi} it is shown that the coefficient $\tilde\beta_0$ in Th.\ \ref{AsympOrbi} is equal to
$$
\tilde\beta_0=\frac{\rk E}2\int_M\frac{c_1(L,h^L)^n}{n!}\log\det\frac{\Omega^L}{2\pi\omega^{TM}}.
$$
Thus one finds when $h^L=h^{'L}$
\begin{eqnarray*}
\lefteqn{
T(M,E\otimes L^p,(g^{'TM},h^{'E},h^{'L}))-T(M,E\otimes L^p,(g^{TM},h^E,h^L))
}\hspace{2cm}\\&=&p^n\frac{\rk E}2\int_M\frac{c_1(L,h^L)^n}{n!}\log\det\frac{g^{TM}}{g^{'TM}}+\mbox{terms of lower order}.
\end{eqnarray*}
The corollary is obtained by combining these two expansions.
\end{proof}
If $\Gamma$ is a finite group and $M/\Gamma$ is an effective global quotient,
i.e. $\Gamma$ acts smoothly effectively on $M$, the orbifold torsion and the equivariant torsion can be computed from each other.
The equivariant torsion and the orbifold torsion are compared in \cite[p.\ 345]{lrr1} in the case where $M/\Gamma$ is smooth, and in \cite[Eq.\ (0.6)]{Ma1} for $\rho$ being the trivial representation.
\begin{lemma}\label{finiteFourier}
Consider a right action by  finite group $\Gamma$ on an Hermitian manifold $M$ by holomorphic isometries. Let $(E,h^E)$ be a $\Gamma$-equivariant holomorphic Hermitian
vector bundle. For a unitary representation $(V_\rho,\rho)$ of $\Gamma$ with character
$\chi_\rho$ let $F_\rho:=M{\times_\rho}V_\rho$ denote the associated flat Hermitian orbifold
vector bundle on $M/\Gamma$.  Then the orbifold holomorphic torsion on $M/\Gamma$ and the equivariant torsion
are Fourier transforms of each other. More precisely,
$$ T(M/\Gamma,\mtr E/\Gamma\otimes
\mtr F_\rho)=\frac1{\#\Gamma}\sum_{g\in\Gamma}\chi_\rho(g) T_g(M,E)
$$ and, equivalently, with the set $\Gamma^\vee$ of irreducible characters $\chi_\rho$
$$ T_g(M,\mtr E)=\sum_{\rho\in\Gamma^\vee} \overline{\chi_\rho(g)} T(M/\Gamma,
E/\Gamma\otimes
 F_\rho)\qquad\forall g\in\Gamma.
$$
\end{lemma}
\begin{proof}
Consider a unitary representation $\rho:\Gamma\to V_\rho$, an equivariant structure with $g^E:E_{xg}\to E_x$ for any $g\in\Gamma$ and the induced action $g:\Gamma(M,E)\to\Gamma(M,E)$, $s\mapsto (x\mapsto g^E(s_{xg}))$. Let
$$
P_\rho:=\frac1{\#\Gamma}\sum_{g\in\Gamma}g^E\otimes \rho(g).
$$
denote the projection on ${\mathfrak A}^{0,q}(M,E)\otimes V_\rho$ to the $\Gamma$-invariant part. Then
\begin{eqnarray*}
P_\rho({\mathfrak A}^{0,q}(M,E)\otimes V_\rho)
&=&
\{s\otimes v\in {\mathfrak A}^{0,q}(M,E)\otimes V_\rho|g^*s\otimes v =
s\otimes\rho(g)v\mbox{ for all } g\in\Gamma \}
\\&=&
 {\mathfrak A}^{0,q}(M/\Gamma,E/\Gamma\otimes F_\rho).
\end{eqnarray*}
Hence when $P^\perp$ denotes the projection on the orthogonal complement of
ker
$\square$, one obtains for any $q$
\begin{eqnarray*}
\Trs \square^{-s} P^\perp_{|{\mathfrak A}^{0,q}(M/\Gamma,E/\Gamma\otimes F_\rho)}
&=&\Trs P_\rho \square^{-s} P^\perp_{|{\mathfrak A}^{0,q}(M,E)\otimes V_\rho}\\
&=&\frac1{\#\Gamma}\sum_{g\in\Gamma} {\rm Tr}\,\rho(g)\cdot\Trs g^*
\square^{-s}P^{\perp}_{|{\mathfrak A}^{0,q}(M,E)}.
\end{eqnarray*}
Thus the first equality in the theorem follows. It is equivalent to the second one by the Schur orthogonality relations (\cite[p. 211, Prop. 2.3]{CR}).
\end{proof}
This result extends to equivariant determinants equipped with Quillen metrics, as the isotypic decomposition shows that
\begin{eqnarray*}
{\mathfrak A}^{0,q}(M,E)&\cong&\bigoplus_{\rho\in\Gamma^\vee}\Hom_\Gamma(V_\rho^*,{\mathfrak A}^{0,q}(M,E))\otimes V_\rho^*
\\&=&\bigoplus_{\rho\in\Gamma^\vee}P_\rho({\mathfrak A}^{0,q}(M,E)\otimes V_\rho)\otimes V_\rho^*
=\bigoplus_{\rho\in\Gamma^\vee}{\mathfrak A}^{0,q}(M/\Gamma,E/\Gamma\otimes F_\rho)\otimes V_\rho^*.
\end{eqnarray*}
\begin{proof}[Proof of Th.\ \ref{AsympEquiv}]
We continue to work with $g$ being an element of a general finite group $\Gamma$ rather than just the cyclic group generated by $g$, since this makes no difference to the proof.
In the case of the orbifold $M/\Gamma$, for the bundles $E/\Gamma\otimes (L/\Gamma)^p\otimes F_\rho$ Th.\ \ref{AsympOrbi} takes the form
\begin{eqnarray*}
\lefteqn{T(M/\Gamma,E/\Gamma\otimes (L/\Gamma)^p\otimes F_\rho)}
\hspace{2cm}\\&=&\frac{\log p}{\#\Gamma}\sum_{h\in\Gamma}\int_{M^h}\left(n\cdot \Td_h(TM)-\Td_h'(TM)\right)\ch_h(E\otimes L^p\otimes V_\rho)
\\&&+\sum_{j=0}^kp^{n-j}\tilde\beta_{j,\rho}+\sum_{j=0}^{k+\nm-n}\sum_{r\in\BZ/(m)}p^{\nm-j}\zeta_m^{r p}\tilde\kappa_{r,j,\rho}+o(p^{n-k})
\end{eqnarray*}
where we can choose $m:=\#\Gamma$ and we added indices for the representation $\rho$. Since $F_\rho$ is flat, $\ch_h(F_\rho)=\chi_\rho(h)$ holds. The characteristic classes depend on $h$ only up to conjugation in $\Gamma$. By the Schur orthogonality relations,
$$\sum_{\rho\in\Gamma^\vee}\overline{\chi_\rho(g)}\chi_\rho(h)=\left\{
\begin{array}{cl}
\#C_\Gamma(g)  &   \mbox{if $g,h$ are conjugate,} \\
0  &  \mbox{else.}
\end{array}
\right.
$$
Thus Lemma \ref{finiteFourier} implies
\begin{eqnarray*}
\lefteqn{T_g(M,E\otimes L^p)}\hspace{2cm}
\\&=&\sum_{\rho\in\Gamma^\vee} \overline{\chi_\rho(g)}\bigg(\frac{\log p}{\#\Gamma}\sum_{h\in\Gamma}\int_{M^h}\left(n\cdot \Td_h(TM)-\Td_h'(TM)\right)\ch_h(E\otimes L^p\otimes V_\rho)
\\&&+\sum_{j=0}^kp^{n-j}\tilde\beta_{j,\rho}+\sum_{j=0}^{k+\nm-n}\sum_{r\in\BZ/(m)}p^{\nm-j}\zeta_m^{r p}\tilde\kappa_{r,j,\rho}\bigg)+o(p^{n-k})
\\&=&\frac{\log p}{\#\Gamma}\sum_{h\in\Gamma}\left(\sum_{\rho\in\Gamma^\vee} \overline{\chi_\rho(g)}\chi_\rho(h)\right)\int_{M^h}\left(n\cdot \Td_h(TM)-\Td_h'(TM)\right)\ch_h(E\otimes L^p)
\\&&+\sum_{\rho\in\Gamma^\vee} \overline{\chi_\rho(g)}\bigg(\sum_{j=0}^kp^{n-j}\tilde\beta_{j,\rho}+\sum_{j=0}^{k+\nm-n}\sum_{r\in\BZ/(m)}p^{\nm-j}\zeta_m^{r p}\tilde\kappa_{r,j,\rho}\bigg)+o(p^{n-k})
\\&=&\log p\cdot\frac{\#C_\Gamma(g)}{\#\Gamma}\sum_{\substack{h{\rm\,conjugate}\\ {\rm to\,}g}}\int_{M^h}\left(n\cdot \Td_h(TM)-\Td_h'(TM)\right)\ch_h(E\otimes L^p)
\\&&+\sum_{\rho\in\Gamma^\vee} \overline{\chi_\rho(g)}\bigg(\sum_{j=0}^kp^{n-j}\tilde\beta_{j,\rho}+\sum_{j=0}^{k+\nm-n}\sum_{r\in\BZ/(m)}p^{\nm-j}\zeta_m^{r p}\tilde\kappa_{r,j,\rho}\bigg)+o(p^{n-k})
\\&=&\log p\cdot\int_{M^g}\left(n\cdot \Td_g(TM)-\Td_g'(TM)\right)\ch_g(E\otimes L^p)
\\&&+\sum_{\rho\in\Gamma^\vee} \overline{\chi_\rho(g)}\bigg(\sum_{j=0}^kp^{n-j}\tilde\beta_{j,\rho}+\sum_{j=0}^{k+\nm-n}\sum_{r\in\BZ/(m)}p^{\nm-j}\zeta_m^{r p}\tilde\kappa_{r,j,\rho}\bigg)+o(p^{n-k}).
\end{eqnarray*}

As detailed after Eq.\ (\ref{FinskiSmooth}), the coefficients $\tilde\alpha_{j,\rho},\tilde \beta_{j,\rho}$ in Th.\ \ref{FinskiMain} take the form $\rk F_\rho$ times a factor which is independent of $F_\rho$.
If $g=e_G$, then the powers of $p$ start at $n=\nm$.
In the case $g\neq e_G$ one finds
$\sum_{\rho\in\Gamma^\vee}\overline{\chi_\rho(g)}\rk F_\rho=0$ and thus $\sum_{\rho\in\Gamma^\vee}\overline{\chi_\rho(g)}\tilde\beta_{j,\rho}=0$.
\end{proof}
It would be interesting to have a general direct proof of Th.\ \ref{AsympEquiv} for general holomorphic isometries $g$. In this equivariant case \cite[Th.\ 8.3]{Bimm} computes the analogue of \cite[Th.\ 2.4]{Ma1}. In the non-equivariant case this formula has been shown in \cite[Th.\ 2.16]{BGS2} (see also \cite[Eq.\ (36)]{BV}). 

Classically (see \cite[1.2-1.3]{GS1} and \cite[p.\ 9]{K3}, which use the opposite sign convention for Bott-Chern classes though) the Bott-Chern forms associated to scaling of the metric by a constant are given by derivatives of the characteristic forms:
\begin{eqnarray*}
\widetilde{\ch}_g(E,c^2h^E,h^E)
&=&\log c^2\cdot\frac\partial{\partial\vep}_{|\vep=0}\sum_\theta\ch\left(-\frac{\Omega^{E_\theta}}{2\pi i}+i\theta+\vep{\rm id}_E\right)=
\log c^2\cdot\ch_g(E,h^E),
\\
\widetilde{\Td}_g(E,c^2h^E,h^E)
&=&
\log c^2\cdot\Td_g'(E,h^E).
\end{eqnarray*}
Hence
\begin{eqnarray*}
\lefteqn{
\log p\cdot\int_{M^g}(n\cdot \Td_g(TM)-\Td_g'(TM))\ch_g(E\otimes L^p)
}\\
&=&\int_{M^g}\Td_g(TM)\widetilde{\ch}_g(E\otimes L^p,p^nh^{E\otimes L^p},h^{E\otimes L^p})
+\int_{M^g}\widetilde{\Td}_g(TX,\frac1p g^{TX},g^{TX})\ch_g(E\otimes L^p).
\end{eqnarray*}
We also obtain an asymptotic anomaly formula like Cor.\ \ref{BV-L2-Orbi}, \cite[Th.\ 10]{BV}, for the equivariant K\"ahler case. As above let $g$ act as a holomorphic isometry on $M$, $L$ and $E$, such that $g^m$ acts as the identity. Then $H^0(M,E\otimes L^p)$ splits into eigenspaces $H^0(M,E\otimes L^p)_{r}$ on which $g$ acts as $\zeta_m^r$ for $r\in\BZ/(m)$. Set $\lambda_{p,r}:=\det H^0(M,E\otimes L^p)_{r}^{-1}$.
Again the Hodge embedding induces $L^2$-metrics $h'_{\lambda,p,r},h_{\lambda,p,r}$ on $\lambda_{p,r}$.
\begin{cor}\label{BV-L2-equiv} 
Let $g^{TM},g^{'TM}$ be two K\"ahler metrics on a compact $n$-dimensional manifold $M$ and let $(L,h^L),(E,h^E)\to M$ be Hermitian vector bundles with $\rk L=1$ and $L$ positive. Let $h^{'E}$ denote a second Hermitian metric on $E$. Assume that $M$, $E$ and $L$ are invariant under the action of an holomorphic isometry $g$ of finite order $m$ and that $E$ and $L$ are equipped with a $g$-equivariant structure.
Then there is a family $\tilde\kappa_\bullet':\BZ/(m)\x\BN_0\to\BR[\zeta_m]$ such that for any $k\in\BN_0$ the $L^2$-metrics on $\lambda_{p,r}$ satisfy for $p\to+\infty$
\begin{align*}
\sum_{r\in\BZ/(m)}\zeta_m^r\log\frac{h'_{\lambda,p,r}}{h_{\lambda,p,r}}=\sum_{j=0}^kp^{\nm-j}\sum_{r\in\BZ/(m)}\zeta_m^{r p}\tilde\kappa_{r,j}'+o(p^{\nm-k}).
\end{align*}
\end{cor}
\begin{proof}
Bismut shows in \cite[Th.\ 0.1, Th.\ 2.5]{Bimm} that
\begin{eqnarray*}
\sum_{r\in\BZ/(m)}\zeta_m^r\log\frac{h'_{\lambda,p,r}}{h_{\lambda,p,r}}
&-&T_g(M,E\otimes L^p,(g^{'TM},h^{'E},h^{'L}))+T_g(M,E\otimes L^p,(g^{TM},h^E,h^L))
\\&=&\int_{M^g}\widetilde\Td_g(TM,g^{TM},g^{'TM})\ch_g(E\otimes L^p,(h^E,h^L))
\\&&+\int_{M^g}\Td_g(TM,g^{'TM})\widetilde\ch_g(E\otimes L^p,(h^E,h^L),(h^{'E},h^{'L})).
\end{eqnarray*}
which is a polynomial in $p$ with coefficients as in the statement
. By Th.\ \ref{AsympEquiv} the difference of the holomorphic torsions also has the form of the right-hand side in the statement.
\end{proof}
\section{An application to Arakelov geometry}
We consider the same setting as in \cite{GS2} and \cite{lrr1}. As in \cite[Def.\ 3.1.1, p.\ 124]{GS2} let a \emph{regular arithmetic ring} $D$ be an excellent regular Noetherian integral ring, together with a finite nonempty set $\cal S$ of ring monomorphisms $D\to\BC$, which is invariant under complex conjugation. By $\mn$ we denote the diagonalisable group scheme over $D$ associated to $\BZ/(m)$. 
Extending \cite[Def.\ 3.2.1]{GS2} an \emph{equivariant arithmetic variety} is defined in \cite[p.\ 357]{lrr1} as a regular integral scheme which is endowed with a $\mn$-projective action over $\Spec\ D$. Let $f:X\to \Spec\ D$ be an equivariant arithmetic variety of dimension $n$. The set $X(\BC)$ of complex points of the variety $X\times_{D}\BC^{\cal S}$ naturally carries the structure of a complex manifold. The group of $m$-th roots of unity acts on $X(\BC)$ by holomorphic automorphisms. We shall write $g$ for the automorphism corresponding to $\zeta_{m}$.
By \cite[Prop.\ 2.12]{lrr1}, the fixed point scheme $X^{\mn}$ is regular and by \cite[Cor.\ 2.11]{lrr1} and the GAGA principle, there are natural isomorphisms of complex manifolds $X^{\mn}(\BC)\simeq X(\BC)^{g}$. Let $f^\mn$ denote the map $X^\mn\to\Spec\ D$ induced by $f$. Complex conjugation induces antiholomorphic automorphisms $F_\infty$ of $X(\BC)$ and $X^{\mu_{m}}(\BC)$. 

An Hermitian equivariant vector bundle on $X$ is a vector bundle $E\to X$ equipped with a $\mu_{m}$-action which lifts the action of $\mu_{m}$ on $X$ and an Hermitian metric $h^E$ on $E_\BC$, the bundle associated to $E$ on the complex points, which is invariant under $F_{\infty}$ and $\mn$. There is a natural orthogonal $\BZ/(m)$-grading $E_{|X^{\mn}}\simeq\oplus_{k\in\BZ/(m)}E_{k}$ on  the restriction of $E$ to the fixed point scheme $X^{\mn}$. 

In \cite{GS2}) Gillet and Soul\'e associate an arithmetic Chow ring $\ar{\rm CH}(X)$ to each such variety, which  carries a natural grading analogous to  the grading of the classical Chow group. For every projective morphism $f:X'\to X$ which is smooth over $\BQ$ between equivariant arithmetic varieties over $D$, there is a push-forward map  $f_{*}:\ar{\rm CH}(X')\to\ar{\rm CH}(X)$.

For an Hermitian bundle $\mtr{E}$ on $X$, Gillet and Soul\'e define arithmetic characteristic classes like an \emph{arithmetic Chern character} $\ar{\ch}(\mtr{E})\in\ar{\rm CH}(X)_\BQ$ and an \emph{arithmetic Todd class} $\ar{\Td}(\mtr{E})\in\ar{\rm CH}(X)_\BQ$.
As in \cite[Def.\ 7.13]{lrr1} we set
$$
\ar\ch_\mn(E,h^E):=\sum_{k\in\BZ/(m)}\ar\ch(E_k,h^E)\otimes \zeta_m^k\in\ar{\rm CH}(X^\mn)\otimes_\BZ\BC.
$$
To $f$ as above Gillet and Soul\'e construct an element $\ar{\Td}(\mtr{Tf})\in\ar{\rm CH}(X)_\BQ$ to the map $f$ and the K\"ahler form $\omega_{X}$. In \cite[p. 393]{lrr1}, using the normal bundle $N$ to $X^\mn\hookrightarrow X$ the \emph{equivariant arithmetic Todd class} is defined as
$$
\ar\Td_\mn(E,h^E):=\left(
\ar\ch_\mn\Big(\sum_{q=0}^{\rk N}(-1)^q\Lambda^qN^\vee,g^{TX}\Big)
\right)^{-1}\ar\Td(Tf^\mn,g^{TX})\in\ar{\rm CH}(X^\mn)\otimes_\BZ\BC..
$$

For simplicity's sake we shall suppose now that $D=\BZ$. Let $X$ be projective and flat over $\Spec\BZ$. 
Denote by $A_{\rm Tors}$ the torsion subgroup of an abelian group $A$.
Gillet and Soul\'e show that there is a natural isomorphism $\ar{\rm deg}:\ar{\rm CH}^{1}(\BZ)\to\BR$, called the \emph{arithmetic degree}.
It is given by $\ar{\rm deg}(\ar{\ch}^{[1]}(V,h^V))=-\log\frac{\covol(V,h^V)}{\#V_{\rm Tors}}$ for a finitely generated Hermitian $\BZ$-module $(V,h^V)$.
In \cite[Th.\ 7.14]{lrr1} an equivariant refinement of Gillet-Soul\'es Riemann-Roch theorem for the arithmetic Chern character and the push-forward map in arithmetic Chow theory (\cite{GS8}) is proven. 
\begin{theor}\label{lrr1} (\cite[Th.\ 7.14]{lrr1})
Let $\mtr{E}$ be an equivariant Hermitian vector bundle 
on $X$. Then the equality
\begin{eqnarray*}\lefteqn{
-\sum_{q\geq 0}(-1)^{q}\sum_{k\in\BZ/(m)}{\zeta_{m}^k}\cdot
\log\frac{\covol(\mtr{H^{q}(X,E)}_k,|\cdot|_{L^2})}
{\#H^{q}(X,E)_{k,{\rm Tors}}}
}\hspace{2cm}\\&=&
\frac12T_{g}(X(\BC),E_\BC)-
\frac12\int_{X^{\mu_{m}}(\BC)}\Td_{g}(TX_\BC)\ch_g(E_\BC)R_{g}
(TX_\BC)
\\&&+
\ar{{\rm deg}}\left( f^\mn_*(\ar{\Td}_{\mn}(\mtr{Tf},g^{TX})
\ar{\ch}_{\mu_{m}}(\mtr{E},h^E)
\right)
\end{eqnarray*}
of complex numbers holds.
\label{BisConj}
\end{theor}
\begin{proof}[Proof of Th.\ \ref{eqHilbert-Samuel}]
Let $X(\BC)^g=:\dot\bigcup_{j\in J}X(\BC)^g_j$ denote the connected components of $X^\mn(\BC)$ of dimension $n_j:=\dim X(\BC)^g_j$. As a function in $p$,
\begin{eqnarray*}
\lefteqn{
\int_{X(\BC)^g}\Td_{g}(TX_\BC)\ch_g(E_\BC\otimes L_\BC^p)R_{g}(TX_\BC)
}
\\&=&\sum_{j\in J}\sum_{r=0}^{n_j}p^r(\ch_g(L_{\BC|X^g_j})^{[0]})^p\int_{X(\BC)^g_j}\Td_{g}(TX_\BC)\frac{c_1(L_\BC)^r}{r!}\ch_g(E_\BC)R_{g}(TX_\BC)
\end{eqnarray*}
is a polynomial in $p$ of degree less or equal to $\nm$ with coefficients of the form $\zeta_m^{p k}$, $k\in\BZ/(m)$. 
Similarly 
one obtains
\begin{eqnarray*}
\lefteqn{
\ar{{\rm deg}}\left( f^\mn_*(\ar{\Td}_{\mn}(\mtr{Tf})
\ar{\ch}_{\mu_{m}}(E\otimes\mtr{L}^p)\right)
}
\\&=&\ar\deg f^\mn_*\left(
\ar\Td_\mn(N,g^{TX})^{[0]}\ar\ch_\mn(E,h^E)^{[0]}\ar\ch_\mn(L)^{[n-\rk N+1]}
\right)+O(p^\nm)
\\&=&\ar\deg f^\mn_*\bigg(\frac{p^{n-\rk N+1}}{(n-\rk N+1)!}\frac{\sum_{\ell\in\BZ/(m)}\zeta_m^{\ell}\rk E_{\ell}}{\prod_{r\in\BZ/(m)}\left(1-\zeta_m^{-r}\right)^{\rk N_r}}
\sum_{k\in\BZ/(m)}\zeta_m^{pk}\ar c_1(L_k,h^L)^{n-\rk N+1}
\bigg)+O(p^\nm)
\end{eqnarray*}
where the $O(p^\nm)$-part is again a polynomial in $p$ with coefficients as above. As $H^0(X,E\otimes L^p)$ has no torsion (since $f$ is flat) and the higher cohomology groups vanish for $p\to+\infty$, Th.\ \ref{lrr1} shows that
\begin{eqnarray*}\lefteqn{
-\sum_{k\in{\BZ}/(m)}{\zeta_m^k}\cdot
\log\covol\,(H^0(X,E\otimes L^p)_k,|\cdot|_{L^2})
}\hspace{2cm}\\&=&
\ar\deg f^\mn_*\bigg(\frac{p^{n-\rk N+1}}{(n-\rk N+1)!}\frac{\sum_{\ell\in\BZ/(m)}\zeta_m^{\ell}\rk E_{\ell}}{\prod_{r\in\BZ/(m)}\left(1-\zeta_m^{-r}\right)^{\rk N_r}}
\sum_{k\in\BZ/(m)}\zeta_m^{pk}\ar c_1(L_k,h^L)^{n-\rk N+1}
\bigg)
\\&&+\frac12T_{g}(X(\BC),E_\BC\otimes L_\BC^p)+\sum_{j=0}^\nm p^{\nm-j}\sum_{r\in\BZ/(m)}\zeta_m^{r p}\hat\kappa'_{r,j}.
\end{eqnarray*}
The asymptotic expansion of $T_g(X(\BC),E_\BC\otimes L_\BC^p)$ from Th.\ \ref{AsympEquiv} provides the remaining part of the right hand side in the statement.
\end{proof}
Since $c_1(L_\BC,h^L)>0$ implies that the free summand of $H^q(X,E\otimes L^p)$ vanishes for $q>0$ and $p>>0$, we could only assume that $H^q(X,E\otimes L^p)_{\rm Tors}=0$ holds for $q>0$ and $p>>0$.

\end{document}